\documentclass[12pt]{article}

\usepackage{amsmath,amssymb}
\usepackage{tikz}
\usepackage{color}
\usepackage[toc]{appendix}
\usepackage{graphicx}
\usepackage{fancyhdr}
\usepackage{enumitem}
\usepackage{bbm}
\usepackage{graphicx}
\usepackage{parskip}
\usepackage{float}
\usepackage{chngpage}
\usepackage{calc}
\usepackage{bigints}
\usepackage{tikz}
\usepackage{array}
\usepackage{booktabs}
\usepackage{rotating}
\usepackage{multirow}
\usepackage{adjustbox}
\usepackage{tabularx}
\usepackage{verbatim}
\usepackage{mathtools}
\usepackage{ragged2e}
\usepackage[makeroom]{cancel}
\usepackage{enumitem}
\usepackage{caption}
\usepackage{hyperref}
\usepackage{caption}
\usepackage{subcaption}
\usepackage{amsmath,amssymb}
\usepackage{appendix}
\usepackage{graphicx}
\usepackage{pgfplots}
\pgfplotsset{compat=1.16}
\usepackage{tikz}
\usepackage{verbatim}
\usepackage{amsthm}
\usepackage[english]{babel}
\usepackage{hyphenat}
\usepackage[makeindex]{imakeidx}
\usepackage{tikz}
\usetikzlibrary{datavisualization}
\usetikzlibrary{matrix}
\usetikzlibrary{datavisualization.formats.functions}

\newtheorem{theorem}{Theorem}[section]

\newtheorem{remark}[theorem]{Remark}

\makeatletter
\def\section{\@startsection {section}{1}{\z@}{3.25ex plus 1ex minus
		.2ex}{1.5ex plus .2ex}{\large\bf}}
\def\subsection{\@startsection{subsection}{2}{\z@}{3.25ex plus 1ex minus
		.2ex}{1.5ex plus .2ex}{\normalsize\bf}}
\@addtoreset{equation}{section} 
\makeatother

\title{A note about the invariance of the basic reproduction number for stochastically perturbed SIS models}

\author{Enrico Bernardi\thanks{Dipartimento di Scienze Statistiche Paolo Fortunati, Università di Bologna, Bologna, Italy. \textbf{e-mail}: enrico.bernardi@unibo.it} \and  Alberto Lanconelli\thanks{Dipartimento di Scienze Statistiche Paolo Fortunati, Università di Bologna, Bologna, Italy. \textbf{e-mail}: alberto.lanconelli2@unibo.it}}

\date{\today}

\begin{document}
	
\maketitle
	
\bigskip
	
\begin{abstract}
We try to justify rigorously, using a Wong-Zakai approximation argument, the susceptible-infected-susceptible (SIS) stochastic differential equation proposed in \cite{Mao 2011}. We discover that according to this approach the \emph{right} stochastic model to be considered should be the Stratonovich version of the It\^o equation analyzed in \cite{Mao 2011}. Surprisingly, this alternative model presents the following feature: the threshold value characterizing the two different asymptotic regimes of the solution coincides with the one describing the classical SIS deterministic equation.
\end{abstract}
	
	
	
\allowdisplaybreaks
	
\section{Introduction}\label{intro}

The susceptible-infected-susceptible (SIS) model is a simple
mathematical model that describes, under suitable assumptions, the
spread of diseases with no permanent immunity (see
e.g. \cite{Brauer},\cite{HY}). In such models an individual starts
being susceptible to a disease, at some point of time gets infected
and then recovers after some other time interval, becoming susceptible
again.
If $S(t)$ and $I(t)$ denote the number of susceptibles and infecteds at time $t$, respectively, then the differential equations describing the spread of the disease are
\begin{align}\label{SIS deterministic}
\begin{cases}
\frac{dS(t)}{dt}=\mu N-\beta S(t)I(t)+\gamma I(t)-\mu S(t),& S(0)=s_0>0;\\
\frac{dI(t)}{dt}=\beta S(t)I(t)-(\mu+\gamma) I(t),& I(0)=i_0>0.
\end{cases}
\end{align}
Here, $N:=s_0+i_0$ is the initial size of the population amongst whom the disease is spreading, $\mu$ denotes the per capita death rate, $\gamma$ is the rate at which infected individuals become cured and $\beta$ stands for the disease transmission coefficient. Note that 
\begin{align*}
\frac{d}{dt}(S(t)+I(t))=\mu(N-(S(t)+I(t))), \quad S(0)+I(0)=N,
\end{align*}
and hence
\begin{align*}
S(t)+I(t)=S(0)+I(0)=N,\quad\mbox{ for all $t\geq 0$}.
\end{align*}
Therefore, system (\ref{SIS deterministic}) reduces to the differential equation
\begin{align}\label{SIS one}
\frac{dI(t)}{dt}=\beta I(t)(N-I(t))-(\mu+\gamma) I(t),\quad I(0)=i_0\in ]0,N[,
\end{align}
with $S(t):=N-I(t)$, for $t\geq 0$. Equation (\ref{SIS one}) can be solved explicitly as
\begin{align}\label{explicit SIS}
I(t)=\frac{i_0e^{[\beta N-(\mu+\gamma)]t}}{1+\beta\int_0^ti_0e^{[\beta N-(\mu+\gamma)]s}ds},\quad t\geq 0,
\end{align}
and one finds that 
\begin{align*}
\lim_{t\to+\infty}I(t)=
\begin{cases}
0,&\mbox{ if $R_0\leq 1$};\\
N(1-1/R_0),&\mbox{ if $R_0>1$},
\end{cases}
\end{align*}
where 
\begin{align*}
R_0:=\frac{\beta N}{\mu+\gamma}.
\end{align*} 
This ratio is known as \emph{basic reproduction number} of the infection and determines whether the disease will become extinct, i.e. $I(t)$ will tend to zero as $t$ goes to infinity, or will be persistent, i.e. $I(t)$ will tend to a positive limit as $t$ increases.

\subsection{The stochastic model}

With the aim of examining the effect of environmental stochasticity, Gray et al. \cite{Mao 2011} have proposed a stochastic version of (\ref{SIS one}) which is obtained via a suitable perturbation of the parameter $\beta$. More precisely, they write equation (\ref{SIS one}) in the differential form
\begin{align}\label{SIS differential}
dI(t)=\beta I(t)(N-I(t))dt-(\mu+\gamma) I(t)dt,\quad I(0)=i_0\in ]0,N[,
\end{align}
and \emph{formally} replace the infinitesimal increment $\beta dt$ with $\beta dt+\sigma dB(t)$, where $\sigma$ is a new positive parameter and $\{B(t)\}_{t\geq 0}$ denotes a standard one dimensional Brownian motion. This perturbation transforms the deterministic differential equation (\ref{SIS one}) into the stochastic differential equation
\begin{align}\label{SDE Mao}
dI(t)=[\beta I(t)(N-I(t))-(\mu+\gamma) I(t)]dt+\sigma I(t)(N-I(t))dB(t),
\end{align}
which the authors interpret in the It\^o's sense. Equation (\ref{SDE Mao}) is then investigated and the authors prove the existence of a unique global strong solution living in the interval $]0,N[$ with probability one for all $t\geq 0$. Moreover, they identify a \emph{stochastic reproduction number} 
\begin{align*}
R_0^S:=R_0-\frac{\sigma^2N^2}{2(\mu+\gamma)},
\end{align*}
which characterizes the following asymptotic behaviour: 
\begin{itemize}
	\item if $R_0^S<1$ and $\sigma^2<\frac{\beta}{N}$ or $\sigma^2>\max\{\frac{\beta}{N},\frac{\beta^2}{2(\mu+\gamma)}\}$, then the \emph{disease will become extinct}, i.e.
	\begin{align*}
		\lim_{t\to+\infty}I(t)=0;
	\end{align*}
	\item if $R_0^S>1$, then \emph{the disease will be persistent}, i.e.
	\begin{align*}
	\liminf_{t\to+\infty}I(t)\leq \xi\leq \limsup_{t\to+\infty}I(t),
	\end{align*}
where $\xi:=\frac{1}{\sigma^2}\left(\sqrt{\beta^2-2\sigma^2(\mu+\gamma)}-(\beta-\sigma^2N)\right)$.
\end{itemize}
It is worth mentioning that Xu \cite{Xu} refined the above description as follows:
\begin{itemize}
	\item if $R_0^S<1$, then $I(t)$ tends to zero, as $t$ tends to infinity, almost surely;
	\item if $R_0^S\geq 1$, then $I(t)$ is recurrent on $]0,N[$.
\end{itemize}

\subsection{The stochastic model revised}

We already mentioned that the It\^o equation 
\begin{align*}
	dI(t)=[\beta I(t)(N-I(t))-(\mu+\gamma) I(t)]dt+\sigma I(t)(N-I(t))dB(t),
\end{align*}
proposed in \cite{Mao 2011} is derived from 
\begin{align}\label{SIS one 1}
	\frac{dI(t)}{dt}=\beta I(t)(N-I(t))-(\mu+\gamma) I(t)
\end{align}
via the formal substitution
\begin{align*}
\beta dt\mapsto \beta dt+\sigma dB(t)
\end{align*}
in 
\begin{align*}
	dI(t)=\beta I(t)(N-I(t))dt-(\mu+\gamma) I(t)dt.
\end{align*}
It is important to remark that the non differentiability of the Brownian paths prevents from the implementation of an otherwise rigorous transformation 
\begin{align}\label{beta perturbation 2}
	\beta \mapsto \beta +\sigma \frac{dB(t)}{dt}
\end{align}
for equation (\ref{SIS one 1}). We now start from this simple observation and try to make such procedure rigorous. \\
Fix $T>0$ and, for a partition $\pi$ of the interval $[0,T]$, let $\{B^{\pi}(t)\}_{t\in [0,T]}$ be the polygonal approximation of the Brownian motion $\{B(t)\}_{t\in [0,T]}$, relative to the partition $\pi$. This means that $\{B^{\pi}(t)\}_{t\in [0,T]}$ is a continuous piecewise linear random function converging to $\{B(t)\}_{t\in [0,T] }$ almost surely and uniformly on $[0,T]$, as the mesh of the partition tends to zero. Now, substituting $\{B(t)\}_{t\in [0,T]}$ with $\{B^{\pi}(t)\}_{t\in [0,T]}$ in (\ref{beta perturbation 2}) we get a well defined transformation
\begin{align*}
	\beta \mapsto \beta +\sigma \frac{dB^{\pi}(t)}{dt},
\end{align*}
which in connection with (\ref{SIS one 1}) leads to the random ordinary differential equation
\begin{align*}
\frac{dI^{\pi}(t)}{dt}=[\beta I^{\pi}(t)(N-I^{\pi}(t))-(\mu+\gamma) I^{\pi}(t)]+\sigma I^{\pi}(t)(N-I^{\pi}(t))\frac{dB^{\pi}(t)}{dt}.
\end{align*}
According to the celebrated Wong-Zakai theorem \cite{WZ}, the solution of the previous equation converges, as the mesh of $\pi$ tends to zero, to the solution $\{\mathtt{I}(t)\}_{t\in [0,T]}$ of the Stratonovich-type stochastic differential equation   
\begin{align}\label{Mao Stratonovich}
d\mathtt{I}(t)=[\beta \mathtt{I}(t)(N-\mathtt{I}(t))-(\mu+\gamma) \mathtt{I}(t)]dt+\sigma \mathtt{I}(t)(N-\mathtt{I}(t))\circ dB(t),
\end{align} 
which is equivalent to the It\^o-type equation
\begin{align}\label{Mao Stratonovich 2}
	d\mathtt{I}(t)=&\left[\beta \mathtt{I}(t)(N-\mathtt{I}(t))-(\mu+\gamma) \mathtt{I}(t)+\frac{\sigma^2}{2}\mathtt{I}(t)(N-\mathtt{I}(t))(N-2\mathtt{I}(t))\right]dt\nonumber\\
	&+\sigma \mathtt{I}(t)(N-\mathtt{I}(t))dB(t)
\end{align} 
(see e.g. \cite{KS} for the definition of Stratonovich integral and It\^o-Stratonovich correction term). Therefore, the model equation obtained via this procedure differs from the one proposed in \cite{Mao 2011} for the presence in the drift coefficient of the additional term
\begin{align*}
\frac{\sigma^2}{2}\mathtt{I}(t)(N-\mathtt{I}(t))(N-2\mathtt{I}(t)).	
\end{align*}
Surprisingly, the \emph{stochastic reproduction number} for the corrected model (\ref{Mao Stratonovich 2}) coincides with $R_0=\frac{\beta N}{\mu+\gamma}$. In other words, the stochastic perturbation of $\beta$ doesn't affect the basic reproduction number.

\begin{theorem}\label{main theorem}
Equation (\ref{Mao Stratonovich 2}) possesses a unique global strong solution $\{\mathtt{I}(t)\}_{t\geq 0}$ which lives in the interval $]0,N[$ for all $t\geq 0$ with probability one. Such solution can be explicitly represented as
\begin{align*}
\mathtt{I}(t)=\frac{i_0\mathcal{E}(t)}{1+\frac{i_0}{N}(\mathcal{E}(t)-1)+i_0\frac{\mu+\gamma}{N}\int_0^t\mathcal{E}(s)ds},\quad t\geq 0,
\end{align*} 
where
\begin{align*}
\mathcal{E}(t):=e^{(\beta N-(\mu+\gamma)) t+N \sigma B(t)}.
\end{align*}
Moreover,
\begin{itemize}
\item if $R_0<1$, then $\mathtt{I}(t)$ tends to zero, as $t$ tends to infinity, almost surely;
\item if $R_0\geq 1$, then $\mathtt{I}(t)$ is recurrent on $]0,N[$.
\end{itemize}	
\end{theorem}

The paper is organized as follows: in Section 2 we develop a general framework to study existence, uniqueness and sufficient conditions for extinction and persistence for a large class of equations which encompasses the model equation (\ref{SDE Mao}) and its revised version (\ref{Mao Stratonovich 2}); Section 3 contains the proof of Theorem \ref{main theorem}. 

\section{A general approach}\label{general method}

Aim of the present section is to propose a general method for studying existence and uniqueness of global strong solutions, as well as conditions for their extinction or persistence, for a large class of equations, which includes (\ref{SDE Mao}) and (\ref{Mao Stratonovich 2}) as particular cases. Namely, we consider stochastic differential equations of the form
\begin{align}\label{SDE intro}
	\begin{cases}
		dX(t)=[f(X(t))-h(X(t))]dt+\sum_{i=1}^mg_i(X(t))dB_i(t),& t>0;\\
		X(0)=x_0\in ]0,N[,&
	\end{cases}
\end{align}
where the coefficients satisfy only those fairly general assumptions needed to derive the desired properties (see Theorem \ref{general existence theorem} below for the detailed assumptions). Our method allows for a great flexibility in the choice of the coefficients while preserving the essential features of (\ref{SDE Mao}) and (\ref{Mao Stratonovich 2}). In particular, we allow the diffusion coefficients to vanish on arbitrary intervals, thus ruling out the techniques based on Feller's test for explosions (see for instance Chapter 5 in \cite{KS}). Also the method based on the Lyapunov function, which is successfully applied in \cite{Mao 2011} doesn't seem to be appropriate for the great generality considered here. Our approach relies instead on two general theorems of the theory of stochastic differential equations, which we now restate for the readers' convenience at the beginning of the next section (see Theorem \ref{boundary} and Theorem \ref{comparison} below). 

Let $(\Omega,\mathcal{F},\mathbb{P})$ be a complete probability space endowed with an $m$-dimensional standard Brownian motion $\{(B_1(t),...,B_m(t))\}_{t\geq 0}$ and denote by $\{\mathcal{F}_t^B\}_{t\geq 0}$ its augmented natural filtration. In the sequel we will be working with one dimensional It\^o's type stochastic differential equations driven by the $m$-dimensional Brownian motion $\{(B_1(t),...,B_m(t))\}_{t\geq 0}$.

\begin{theorem}\label{boundary}
Let $\{X(t)\}_{t\geq 0}$ be the unique global strong solution of the stochastic differential equation
\begin{align*}
\begin{cases}
dX(t)=\mu(X(t))dt+\sum_{i=1}^m\sigma_i(X(t))dB_i(t),& t>0;\\
X(0)=x_0\in \mathbb{R},&
\end{cases}
\end{align*}
where the coefficients $\mu,\sigma_1,...,\sigma_m:\mathbb{R}\to\mathbb{R}$ are assumed to be globally Lipschitz continuous. If we set
\begin{align*}
\Lambda:=\{x\in\mathbb{R}:\mu(x)=\sigma_1(x)=\cdot\cdot\cdot=\sigma_m(x)=0\}
\end{align*}
and assume $x_0\notin\Lambda$, then
\begin{align*}
\mathbb{P}\left(X(t)\notin\Lambda,\mbox{ for all }t\geq 0\right)=1.
\end{align*}
\end{theorem}

\begin{proof}
See the theorem in \cite{L}.
\end{proof}

\begin{theorem}\label{comparison}
Let $\{X(t)\}_{t\geq 0}$ be the unique global strong solution of the stochastic differential equation
\begin{align*}
\begin{cases}
dX(t)=\mu_1(X(t))dt+\sum_{i=1}^m\sigma_i(X(t))dB_i(t),& t>0;\\
X(0)=z\in \mathbb{R},&
\end{cases}
\end{align*}
and $\{Y(t)\}_{t\geq 0}$ be the unique global strong solution of the stochastic differential equation
\begin{align*}
\begin{cases}
dY(t)=\mu_2(Y(t))dt+\sum_{i=1}^m\sigma_i(Y(t))dB_i(t),& t>0;\\
Y(0)=z\in \mathbb{R},&
\end{cases}
\end{align*}
where the coefficients $\mu_1,\mu_2,\sigma_1,...,\sigma_m:\mathbb{R}\to\mathbb{R}$ are assumed to be globally Lipschitz continuous. If $\mu_1(z)\leq\mu_2(z)$, for all $z\in\mathbb{R}$, then
\begin{align*}
\mathbb{P}\left(X(t)\leq Y(t),\mbox{ for all }t\geq 0\right)=1.
\end{align*}
\end{theorem}

\begin{proof}
	See Proposition 2.18, Chapter 5 in \cite{KS}, where the proof is given for $m=1$. The extension to several Brownian motions is immediate. See also Theorem 1.1, Chapter VI in \cite{Ikeda Watanabe}.   
\end{proof}

\subsection{Existence, uniqueness and support}

We are now ready to state our existence and uniqueness result.

\begin{theorem}\label{general existence theorem}
	For $i\in\{1,...,m\}$, let $f,g_i,h:\mathbb{R}\to\mathbb{R}$ be locally Lipschitz-continuous functions such that 
	\begin{enumerate}
		\item $f(0)=g_i(0)=0$ and $f(N)=g_i(N)=0$, for some $N>0$;
		\item $h(0)=0$ and $h(x)>0$, when $x>0$.
	\end{enumerate}
	Then, the stochastic differential equation
	\begin{align}
	\begin{cases}\label{SDE many BM}
	dX(t)=[f(X(t))-h(X(t))]dt+\sum_{i=1}^mg_i(X(t))dB_i(t),& t>0;\\
	X(0)=x_0\in ]0,N[,&
	\end{cases}
	\end{align}
	admits a unique global strong solution, which satisfies $\mathbb{P}(0<X(t)<N)=1$, for all $t\geq 0$.
\end{theorem}

\begin{remark}
	It is immediate to verify that equations (\ref{SDE Mao}) and (\ref{Mao Stratonovich 2}) fulfill the assumptions of Theorem \ref{general existence theorem}. 	
\end{remark}

\begin{proof}
The local Lipschitz-continuity of the coefficients entails pathwise uniqueness for  equation (\ref{SDE many BM}), see for instance Theorem 2.5, Chapter 5 in \cite{KS}. Now, we consider the modified equation
\begin{align}
\begin{cases}\label{SDE 3}
d\mathcal{X}(t)=[\bar{f}(\mathcal{X}(t))-\hat{h}(\mathcal{X}(t))]dt+\sum_{i=1}^m\bar{g}_i(\mathcal{X}(t))dB_i(t),& t>0;\\
\mathcal{X}(0)=x_0\in ]0,N[,&
\end{cases}
\end{align}
where 
\begin{align*}
\bar{f}(x)=
\begin{cases}
f(x),&\mbox{if $x\in [0,N]$};\\
0,&\mbox{if $x\notin [0,N]$},
\end{cases}
\quad \mbox{ and }\quad
\bar{g}_i(x)=
\begin{cases}
g_i(x),&\mbox{if $x\in [0,N]$};\\
0,&\mbox{if $x\notin [0,N]$},
\end{cases}
\end{align*}
while
\begin{align*}
\hat{h}(x)=
\begin{cases}
0,&\mbox{if $x<0$};\\
h(x),&\mbox{if $x\in [0,N]$};\\
h(N),&\mbox{if $x>N$}.
\end{cases}
\end{align*}
The coefficients of equation (\ref{SDE 3}) are bounded and globally Lipschitz-continuous; this implies the existence of a unique global strong solution $\{\mathcal{X}(t)\}_{t\geq 0}$ for (\ref{SDE 3}). Moreover, the drift and diffusion coefficients vanish at $x=0$. Therefore, according to Theorem \ref{boundary}, the solution never visits the origin, unless it starts from there. Since $\mathcal{X}(0)=x_0\in ]0,N[$, we deduce that $\mathcal{X}(t)> 0$, for all $t\geq 0$, almost surely. Recalling the assumption $h(x)>0$ for $x>0$, we can rewrite equation (\ref{SDE 3}) as
\begin{align}
\begin{cases}\label{SDE 4}
d\mathcal{X}(t)=[\bar{f}(\mathcal{X}(t))-\hat{h}(\mathcal{X}(t))^+]dt+\sum_{i=1}^m\bar{g}_i(\mathcal{X}(t))dB_i(t),& t>0;\\
\mathcal{X}(0)=x_0\in ]0,N[,&
\end{cases}
\end{align}
where $x^+:=\max\{x,0\}$. We now compare the solution of the previous equation with the one of  
\begin{align}
\begin{cases}\label{SDE 5}
d\mathcal{Y}(t)=\bar{f}(\mathcal{Y}(t))dt+\sum_{i=1}^m\bar{g}_i(\mathcal{Y}(t))dB_i(t),& t>0;\\
\mathcal{Y}(0)=x_0\in ]0,N[,&
\end{cases}
\end{align}
which also possesses a unique global strong solution $\{\mathcal{Y}(t)\}_{t\geq 0}$. Systems (\ref{SDE 4}) and (\ref{SDE 5}) have the same initial condition and diffusion coefficients; moreover, the drift in (\ref{SDE 5}) is greater than the drift in (\ref{SDE 4}). By Theorem \ref{comparison} we conclude that 
\begin{align*}
\mathcal{X}(t)\leq \mathcal{Y}(t),\quad\mbox{ for all $t\geq 0$},
\end{align*} 
almost surely. Moreover, both the drift and diffusion coefficients in (\ref{SDE 5}) vanish at $x=N$. Therefore, invoking once more Theorem \ref{boundary}, the solution never visits $N$, unless it starts from there. Since $\mathcal{Y}(0)=x_0\in ]0,N[$, we deduce that $\mathcal{Y}(t)< N$, for all $t\geq 0$, almost surely. Combining all these facts, we conclude that 
\begin{align*}
0<\mathcal{X}(t)< N,\quad\mbox{ for all $t\geq 0$},
\end{align*}
almost surely. This in turn implies
\begin{align*}
\bar{f}(\mathcal{X}(t))=f(\mathcal{X}(t)),\quad \bar{g}_i(\mathcal{X}(t))=g_i(\mathcal{X}(t)),\quad
\hat{h}(\mathcal{X}(t))=h(\mathcal{X}(t)),
\end{align*}
and that $\{\mathcal{X}(t)\}_{t\geq 0}$ solves equation
\begin{align*}
\begin{cases}
d\mathcal{X}(t)=[f(\mathcal{X}(t))-h(\mathcal{X}(t))]dt+\sum_{i=1}^mg_i(\mathcal{X}(t))dB(t),& t>0;\\
\mathcal{X}(0)=x_0\in ]0,N[,&
\end{cases}
\end{align*}
which coincides with (\ref{SDE many BM}). The uniqueness of the solution completes the proof.
\end{proof}

\subsection{Extinction}

We now investigate the asymptotic behaviour of the solution of (\ref{SDE many BM}); here we are interested in sufficient conditions for extinction.

\begin{theorem}\label{theorem extintion}
Under the same assumptions of Theorem \ref{general existence theorem} assuming in addition,
\begin{align}\label{sup}
\sup_{x\in ]0,N[}\left\{\frac{f(x)-h(x)}{x}-\frac{1}{2}\sum_{i=1}^m\frac{g^2_i(x)}{x^2}\right\}<0,
\end{align}
the solution $\{X(t)\}_{t\geq 0}$ of equation (\ref{SDE many BM}) tends to zero exponentially, as $t$ tends to infinity, almost surely. More precisely,
\begin{align*}
\limsup_{t\to+\infty}\frac{\ln(X(t))}{t}\leq\sup_{x\in ]0,N[}\left\{\frac{f(x)-h(x)}{x}-\frac{1}{2}\sum_{i=1}^m\frac{g^2_i(x)}{x^2}\right\}<0,
\quad\mbox{ almost surely},
\end{align*}
\end{theorem}

\begin{proof}	
We follow the proof of Theorem 4.1 in \cite{Mao 2011}. First of all, we observe that the local Lipschitz-continuity of $f$ implies the existence of a constant $L_N$ such that
\begin{align*}
|f(x)-f(0)|\leq L_N|x-0|,\quad \mbox{ for all $x\in [0,N]$}.
\end{align*}
In particular, using the equality $f(0)=0$, we can rewrite the previous condition as
\begin{align*}
\left|\frac{f(x)}{x}\right|\leq L_N,\quad \mbox{ for all $x\in [0,N]$}.
\end{align*}
Since the same reasoning applies also to $h$ and $g_i$, for $i\in\{1,...,m\}$, we deduce that the supremum in (\ref{sup}) is always finite. \\
Now, let $\{X(t)\}_{t\geq 0}$ be the unique global strong solution of equation (\ref{SDE many BM}). An application of the It\^o formula gives
\begin{align}\label{ito}
\ln(X(t))=&\ln(x_0)+\int_0^t\left[\frac{f(X(s))-h(X(s))}{X(s)}-\frac{1}{2}\sum_{i=1}^m\frac{g^2_i(X(s))}{X(s)^2}\right]ds\\
&+\sum_{i=1}^m\int_0^t\frac{g_i(X(s))}{X(s)}dB_i(s)\nonumber.
\end{align}
Note that the boundedness of the function $x\mapsto \frac{g_i(x)}{x}$ on $]0,N[$ mentioned above entails that the stochastic process
\begin{align*}
t\mapsto \sum_{i=1}^m\int_0^t\frac{g_i(X(s))}{X(s)}dB_i(s),\quad t\geq 0,
\end{align*}
is an $(\{\mathcal{F}_t^B\}_{t\geq 0},\mathbb{P})$-martingale. Therefore, from the strong law of large numbers for martingales (see e.g. Theorem 3.4, Chapter 1 in \cite{Mao book}) we conclude that
\begin{align*}
\lim_{t\to +\infty}\sum_{i=1}^m\frac{1}{t}\int_0^t\frac{g_i(X(s))}{X(s)}dB_i(s)=0,
\end{align*}
almost surely. This fact, combined with (\ref{ito}) gives
\begin{align*}
\limsup_{t\to+\infty}\frac{\ln(X(t))}{t}\leq&\limsup_{t\to+\infty}\frac{1}{t}\int_0^t\left[\frac{f(X(s))-h(X(s))}{X(s)}-\frac{1}{2}\sum_{i=1}^m\frac{g^2_i(X(s))}{X(s)^2}\right]ds\\
\leq&\limsup_{t\to+\infty}\sup_{x\in ]0,N[}\left\{\frac{f(x)-h(x)}{x}-\frac{1}{2}\sum_{i=1}^m\frac{g^2_i(x)}{x^2}\right\}<0,
\end{align*}
almost surely. The proof is complete.
\end{proof}

\subsection{Persistence}

We now search for conditions ensuring the persistence for the solution $\{X(t)\}_{t\geq 0}$ of (\ref{SDE many BM}).

\begin{theorem}\label{persistence}
	Under the same assumptions of Theorem \ref{general existence theorem}, if inequality 
	\begin{align}\label{sup 2}
	\sup_{x\in ]0,N[}\left\{\frac{f(x)-h(x)}{x}-\frac{1}{2}\sum_{i=1}^m\frac{g^2_i(x)}{x^2}\right\}>0,
	\end{align}
	holds and moreover the function
	\begin{align}\label{function}
	x\mapsto\frac{f(x)-h(x)}{x}-\frac{1}{2}\sum_{i=1}^m\frac{g^2_i(x)}{x^2}
	\end{align}
	is strictly decreasing on the interval $]0,N[$, then the solution $\{X(t)\}_{t\geq 0}$ of the stochastic differential equation (\ref{SDE many BM}) verifies
	\begin{align}\label{thesis}
	\limsup_{t\to+\infty}X(t)\geq\xi\quad\mbox{ and }\quad\liminf_{t\to+\infty}X(t)\leq\xi,
	\end{align}
	almost surely. Here, $\xi$ is the unique zero of the function (\ref{function}) in the interval $[0,N]$. 
\end{theorem}

\begin{proof}
We follow the proof of Theorem 5.1 in \cite{Mao 2011}. To ease the notation we set
\begin{align*}
\eta(x):=\frac{f(x)-h(x)}{x}-\frac{1}{2}\sum_{i=1}^m\frac{g^2_i(x)}{x^2},\quad x\in [0,N].
\end{align*}
First of all, we note that $\eta(N)=-\frac{h(N)}{N}<0$; this gives, in combination with (\ref{sup 2}) and the strict monotonicity of $\eta$, the existence and uniqueness of $\xi$. Now, assume the first inequality in (\ref{thesis}) to be false. This implies the existence of $\varepsilon>0$ such that 
\begin{align}\label{A}
\mathbb{P}\left(\limsup_{t\to+\infty}X(t)\leq\xi-2\varepsilon\right)>\varepsilon.
\end{align} 
In particular, for any $\omega\in A:=\{\limsup_{t\to+\infty}X(t)\leq\xi-2\varepsilon\}$, there exists $T(\omega)$ such that
\begin{align*}
X(t,\omega)\leq \xi-\varepsilon,\quad\mbox{ for all }t\geq T(\omega),
\end{align*}
which implies
\begin{align*}
\eta(X(t,\omega))\geq \eta(\xi-\varepsilon)>0,\quad\mbox{ for all }\omega\in A\mbox{ and } t\geq T(\omega).
\end{align*}
Therefore, for $\omega\in A$ and $t> T(\omega)$ we can write
\begin{align*}
\frac{\ln(X(t))}{t}=&\frac{\ln(x_0)}{t}+\frac{1}{t}\int_0^t\eta(X(s))ds+\sum_{i=1}^m\frac{1}{t}\int_0^t\frac{g_i(X(s))}{I(s)}dB_i(s)\\
=&\frac{\ln(x_0)}{t}+\frac{1}{t}\int_0^{T(\omega)}\eta(X(s))ds+\frac{1}{t}\int_{T(\omega)}^t\eta(X(s))ds\\
&+\sum_{i=1}^m\frac{1}{t}\int_0^t\frac{g_i(X(s))}{X(s)}dB_i(s)\\
\geq&\frac{\ln(x_0)}{t}+\frac{1}{t}\int_0^{T(\omega)}\eta(X(s))ds+\frac{t-T(\omega)}{t}\eta(\xi-\varepsilon)\\
&+\sum_{i=1}^m\frac{1}{t}\int_0^t\frac{g_i(X(s))}{X(s)}dB_i(s).
\end{align*}
Hence, recalling that the strong law of large numbers for martingales gives
\begin{align*}
\lim_{t\to +\infty}\sum_{i=1}^m\frac{1}{t}\int_0^t\frac{g_i(X(s))}{X(s)}dB_i(s)=0\quad\mbox{almost surely},
\end{align*}
we conclude that
\begin{align*}
\liminf_{t\to +\infty}\frac{\ln(X(t))}{t}\geq\eta(\xi-\varepsilon)>0,\quad\mbox{on the set $A$},
\end{align*}
which implies
\begin{align*}
\lim_{t\to+\infty}X(t)=+\infty,\quad\mbox{on the set $A$}.
\end{align*}
This contradicts (\ref{A}) and hence prove the first inequality in (\ref{thesis}). \\
The second inequality in (\ref{thesis}) is proven similarly; if the thesis is not true, then  
\begin{align}\label{B}
\mathbb{P}\left(\liminf_{t\to+\infty}X(t)\geq\xi+2\varepsilon\right)>\varepsilon.
\end{align} 
for some positive $\varepsilon$. In particular, for any $\omega\in B:=\{\liminf_{t\to+\infty}X(t)\geq\xi+2\varepsilon\}$, there exists $S(\omega)$ such that
\begin{align*}
X(t,\omega)\geq \xi+\varepsilon,\quad\mbox{ for all }t\geq S(\omega),
\end{align*}
which implies
\begin{align*}
\eta(X(t,\omega))\leq \gamma(\xi+\varepsilon)<0,\quad\mbox{ for all }t\geq S(\omega).
\end{align*}
Therefore, for $\omega\in B$ and $t> S(\omega)$ we can write
\begin{align*}
\frac{\ln(X(t))}{t}=&\frac{\ln(x_0)}{t}+\frac{1}{t}\int_0^t\eta(X(s))ds+\sum_{i=1}^m\frac{1}{t}\int_0^t\frac{g_i(X(s))}{X(s)}dB_i(s)\\
=&\frac{\ln(x_0)}{t}+\frac{1}{t}\int_0^{S(\omega)}\eta(X(s))ds+\frac{1}{t}\int_{S(\omega)}^t\eta(X(s))ds\\
&+\sum_{i=1}^m\frac{1}{t}\int_0^t\frac{g_i(X(s))}{X(s)}dB_i(s)\\
\leq&\frac{\ln(x_0)}{t}+\frac{1}{t}\int_0^{S(\omega)}\eta(X(s))ds+\frac{t-S(\omega)}{t}\eta(\xi+\varepsilon)\\
&+\sum_{i=1}^m\frac{1}{t}\int_0^t\frac{g_i(X(s))}{X(s)}dB_i(s)\\
\end{align*}
Therefore,
\begin{align*}
\limsup_{t\to +\infty}\frac{\ln(X(t))}{t}\leq\eta(\xi+\varepsilon)<0,\quad\mbox{on the set $B$},
\end{align*}
which implies
\begin{align*}
\lim_{t\to+\infty}X(t)=0,\quad\mbox{on the set $B$}.
\end{align*}
This contradicts (\ref{B}) and hence proves the second inequality in (\ref{thesis}).
\end{proof}

\section{Proof of Theorem \ref{main theorem}}

We are now ready to prove our main theorem.

\subsection{Existence, uniqueness, extinction and persistence}

It is immediate to verify that the SDE
\begin{align}\label{Mao Stratonovich 3}
	d\mathtt{I}(t)=&\left[\beta \mathtt{I}(t)(N-\mathtt{I}(t))-(\mu+\gamma) \mathtt{I}(t)+\frac{\sigma^2}{2}\mathtt{I}(t)(N-\mathtt{I}(t))(N-2\mathtt{I}(t))\right]dt\nonumber\\
	&+\sigma \mathtt{I}(t)(N-\mathtt{I}(t))dB(t),
\end{align}     
with initial condition $\mathtt{I}(0)=i_0\in ]0,N[$ fulfills the assumptions of Theorem \ref{general existence theorem} if we set $m=1$,  
\begin{align*}
f(x):=\beta x(N-x)+\frac{\sigma^2}{2}x(N-x)(N-2x),\quad h(x)=(\mu+\gamma)x,\quad g(x):=\sigma x(N-x).
\end{align*}
Therefore, equation (\ref{Mao Stratonovich 3}) possesses a unique global strong solution which lives in the interval $]0,N[$ for all $t\geq 0$ with probability one.\\
Let us now observe that
\begin{align*}
\eta(x)&:=\frac{f(x)-h(x)}{x}-\frac{1}{2}\frac{g^2(x)}{x^2}\\
&=\beta (N-x)+\frac{\sigma^2}{2}(N-x)(N-2x)-(\mu+\gamma)-\frac{1}{2}\sigma^2 (N-x)^2\\
&=\left(\frac{\sigma^2}{2}x-\beta\right)(x-N)-(\mu+\gamma),
\end{align*}
and hence
\begin{align*}
\eta(0)=\beta N-(\mu+\gamma)\quad\mbox{ and }\quad \eta(N)=-(\mu+\gamma).
\end{align*}
This gives:
\begin{itemize}
	\item if $\beta N-(\mu+\gamma)<0$, that means $\frac{\beta N}{\mu+\gamma}<1$, then the assumptions of Theorem \ref{theorem extintion} are fulfilled ($\gamma$ is a convex second order polynomial which takes negative values on the boundaries of $[0,N]$); therefore,  $\mathtt{I}(t)$ will be extinct as $t$ tends to infinity;
	\item if $\beta N-(\mu+\gamma)>0$, that means $\frac{\beta N}{\mu+\gamma}>1$, then the assumptions of Theorem \ref{persistence} are fulfilled ($\gamma$ is a convex second order polynomial which takes a positive value at $0$ and a negative value at $N$); therefore, $\mathtt{I}(t)$ will be persistent as $t$ tends to infinity.
\end{itemize}

\subsection{Explicit representation of the solution}

We observe that the solution of the deterministic equation 
\begin{align}\label{a}
\frac{dI(t)}{dt}=\beta(t) I(t)(N-I(t))-(\mu+\gamma) I(t),\quad I(0)=i_0\in ]0,N[,
\end{align}
where $t\mapsto\beta(t)$ is now a continuous function of $t$, can be written as 
\begin{align}\label{b}
	I(t)=\frac{i_0e^{\int_0^tN\beta(s)ds-(\mu+\gamma)t}}{1+\int_0^t\beta(s)i_0e^{\int_0^sN\beta(r)dr-(\mu+\gamma)s}ds},\quad t\geq 0.
\end{align}
If we set $\beta(t):=\beta+\sigma\dot{B}^{\pi}(t)$, where $\dot{B}^{\pi}(t)$ stands for $\frac{d}{dt}B^{\pi}(t)$, then equation (\ref{a}) and formula (\ref{b}) become respectively
\begin{align}\label{a1}
\frac{dI^{\pi}(t)}{dt}=\beta I^{\pi}(t)(N-I^{\pi}(t))-(\mu+\gamma) I^{\pi}(t)+\sigma I^{\pi}(t)(N-I^{\pi}(t))\dot{B}^{\pi}(t),
\end{align}
with initial condition $I^{\pi}(0)=i_0\in ]0,N[$, and 
\begin{align}\label{b1}
	I^{\pi}(t)=\frac{i_0e^{\int_0^tN\left(\beta+\sigma\dot{B}^{\pi}(s)\right)ds-(\mu+\gamma)t}}{1+\int_0^t(\beta+\sigma\dot{B}^{\pi}(s))i_0e^{\int_0^sN(\beta+\sigma\dot{B}^{\pi}(r))dr-(\mu+\gamma)s}ds}.
\end{align}	
We recall that according to the Wong-Zakai theorem the stochastic process $\{I^{\pi}(t)\}_{t\geq 0}$ converges, as the mesh of the partition $\pi$ tends to zero, to the unique strong solution of the Stratonovich SDE 
\begin{align*}
d\mathtt{I}(t)=[\beta \mathtt{I}(t)(N-\mathtt{I}(t))-(\mu+\gamma) \mathtt{I}(t)]dt+\mathtt{I}(t)\sigma (N-\mathtt{I}(t))\circ dB(t),\quad \mathtt{I}(0)=i_0\in ]0,N[,
\end{align*} 
which is equivalent to the It\^o-type equation
\begin{align}\label{c}
d\mathtt{I}(t)=&\left[\beta \mathtt{I}(t)(N-\mathtt{I}(t))-(\mu+\gamma) \mathtt{I}(t)+\frac{\sigma^2}{2}\mathtt{I}(t)(N-\mathtt{I}(t))(N-2\mathtt{I}(t))\right]dt\nonumber\\
&+\sigma \mathtt{I}(t)(N-\mathtt{I}(t))dB(t),\quad \mathtt{I}(0)=i_0\in ]0,N[.
\end{align} 
We now simplify the expression in (\ref{b1}) and compute its limit as the mesh of the partition $\pi$ tends to zero: this will give us an explicit representation for the solution of (\ref{c}). To ease the notation we set
\begin{align*}
\mathcal{E}^{\pi}(t):=e^{\int_0^tN\left(\beta+\sigma\dot{B}^{\pi}(s)\right)ds-(\mu+\gamma)t}=e^{N\beta t+N \sigma B^{\pi}(t)-(\mu+\gamma)t}=e^{\delta t+N \sigma B^{\pi}(t)},
\end{align*}
where $\delta:=N\beta-(\mu+\gamma)$, and rewrite (\ref{b1}) as
\begin{align}\label{d}	
I^{\pi}(t)&=\frac{i_0\mathcal{E}^{\pi}(t)}{1+i_0\int_0^t(\beta+\sigma\dot{B}^{\pi}(s))\mathcal{E}^{\pi}(s)ds}\nonumber\\
&=\frac{i_0\mathcal{E}^{\pi}(t)}{1+i_0\beta\int_0^t\mathcal{E}^{\pi}(s)ds+ i_0\sigma\int_0^t\dot{B}^{\pi}(s)\mathcal{E}^{\pi}(s)ds}.
\end{align}	
Note that $\delta\geq 0$ if and only if $R_0=\frac{\beta N}{\mu+\gamma}\geq 1$. Now, consider the second integral in the denominator above: an integration by parts gives
\begin{align*}
\int_0^t\dot{B}^{\pi}(s)\mathcal{E}^{\pi}(s)ds&=\int_0^t\dot{B}^{\pi}(s)e^{\delta s+N \sigma B^{\pi}(s)}ds\\
&=\int_0^t\dot{B}^{\pi}(s)e^{N\sigma B^{\pi}(s)}e^{\delta s}ds\\
&=\frac{1}{N\sigma}\left(e^{N \sigma B^{\pi}(t)}e^{\delta t}-1\right)-\frac{\delta}{N\sigma}\int_0^te^{N\sigma B^{\pi}(s)}e^{\delta s}ds\\
&=\frac{1}{N\sigma}(\mathcal{E}^{\pi}(t)-1)-\frac{\delta}{N\sigma}\int_0^t\mathcal{E}^{\pi}(s)ds.
\end{align*}
Therefore, inserting the last expression in (\ref{d}) we get
\begin{align*}	
I^{\pi}(t)&=\frac{i_0\mathcal{E}^{\pi}(t)}{1+i_0\beta\int_0^t\mathcal{E}^{\pi}(s)ds+ i_0\sigma\int_0^t\dot{B}^{\pi}(s)\mathcal{E}^{\pi}(s)ds}\\
&=\frac{i_0\mathcal{E}^{\pi}(t)}{1+i_0\beta\int_0^t\mathcal{E}^{\pi}(s)ds+\frac{i_0}{N}(\mathcal{E}^{\pi}(t)-1)-\frac{i_0\delta}{N}\int_0^t\mathcal{E}^{\pi}(s)ds}\\
&=\frac{i_0\mathcal{E}^{\pi}(t)}{1+\frac{i_0}{N}(\mathcal{E}^{\pi}(t)-1)+i_0\left(\beta-\frac{\delta}{N}\right)\int_0^t\mathcal{E}^{\pi}(s)ds}\\
&=\frac{i_0\mathcal{E}^{\pi}(t)}{1+\frac{i_0}{N}(\mathcal{E}^{\pi}(t)-1)+i_0\frac{\mu+\gamma}{N}\int_0^t\mathcal{E}^{\pi}(s)ds}.
\end{align*}	
We can now let the mesh of the partition $\pi$ tend to zero and get
\begin{align*}
\mathtt{I}(t)&=\lim_{|\pi|\to 0}I^{\pi}(t)=\lim_{|\pi|\to 0}\frac{i_0\mathcal{E}^{\pi}(t)}{1+\frac{i_0}{N}(\mathcal{E}^{\pi}(t)-1)+i_0\frac{\mu+\gamma}{N}\int_0^t\mathcal{E}^{\pi}(s)ds}\\
&=\frac{i_0\mathcal{E}(t)}{1+\frac{i_0}{N}(\mathcal{E}(t)-1)+i_0\frac{\mu+\gamma}{N}\int_0^t\mathcal{E}(s)ds},
\end{align*} 
with
\begin{align*}
\mathcal{E}(t):=e^{\delta t+N \sigma B(t)}.
\end{align*}

\subsection{Recurrence}

To prove the recurrence of $\mathtt{I}(t)$ in the case $R_0\geq 1$, we need to exploit the specific structure of equation (\ref{Mao Stratonovich 3}). In particular, we will follow the approach utilized in \cite{Xu} which is based on the Feller's test for explosion (see for instance Chapter 5.5 C in \cite{KS}).\\
Let $\varphi(x):=\ln\left(\frac{x}{N-x}\right)$ and apply the It\^o formula to $\varphi(\mathtt{I}(t))$; this gives
\begin{align*}
d\varphi(\mathtt{I}(t))&=\left(\beta N-(\mu+\gamma)-(\mu+\gamma)\frac{\mathtt{I}(t)}{N-\mathtt{I}(t)}\right)dt+\sigma NdB(t)\\
&=\left(\beta N-(\mu+\gamma)-(\mu+\gamma)e^{\varphi(\mathtt{I}(t))}\right)dt+\sigma NdB(t),
\end{align*}
and, setting $\mathtt{J}(t):=\varphi(\mathtt{I}(t))$, we can write
\begin{align*}
d\mathtt{J}(t)=\left(\beta N-(\mu+\gamma)-(\mu+\gamma)e^{\mathtt{J}(t)}\right)dt+\sigma NdB(t).
\end{align*} 
Now, the scale function for this process is
\begin{align*}
\psi(x)=\int_0^x\theta(y)dy
\end{align*}
where
\begin{align*}
\theta(y)&=\exp\left\{-\frac{2}{\sigma^2N^2}\int_0^y\beta N-(\mu+\gamma)-(\mu+\gamma)e^{z}dz\right\}\\
&=\exp\left\{-\frac{2(\beta N-(\mu+\gamma))}{\sigma^2N^2}y+\frac{2(\mu+\gamma)}{\sigma^2N^2}(e^y-1)\right\}.
\end{align*}
It is clear that $\psi(+\infty)=+\infty$; moreover, for $\beta N-(\mu+\gamma)\geq 0$, that means $R_0\geq 1$, we get $\psi(-\infty)=-\infty$. These two facts together with Proposition 5.22, Chapter 5 in \cite{KS} imply that $\{\mathtt{J}(t)\}_{t\geq 0}$ is recurrent on $]-\infty, +\infty[$ and hence that $\{\mathtt{I}(t)\}_{t\geq 0}$ is recurrent on $]0,N[$.

\end{document}